\documentclass{amsart}%
\usepackage{amsfonts}
\usepackage{amsmath}
\usepackage{amssymb}
\usepackage{graphicx}%
\newtheorem{theorem}{Theorem}[section]

\theoremstyle{definition}
\newtheorem{definition}[theorem]{Definition}
\newtheorem{example}[theorem]{Example}

\theoremstyle{remark}
\newtheorem{remark}[theorem]{Remark}
\numberwithin{equation}{section}
\theoremstyle{plain}
\newtheorem{acknowledgement}{Acknowledgement}

\newtheorem{conjecture}{Conjecture}
\newtheorem{corollary}{Corollary}

\newtheorem{problem}{Problem}
\newtheorem{proposition}{Proposition}

\copyrightinfo{2014}{David Glickenstein}
\begin{document}
\title[Discrete uniformization of the disk]{ Euclidean formulation of discrete uniformization of the disk}
\author{David Glickenstein}
\address{University of Arizona}
\email{glickenstein@math.arizona.edu}
\urladdr{http://math.arizona.edu/\symbol{126}glickenstein}

\begin{abstract}
Thurston's circle packing approximation of the Riemann Mapping (proven to give
the Riemann Mapping in the limit by Rodin-Sullivan) is largely based on the
theorem that any topological disk with a circle packing metric can be deformed
into a circle packing metric in the disk with boundary circles internally
tangent to the circle. The main proofs of the uniformization use hyperbolic
volumes (Andreev) or hyperbolic circle packings (by Beardon and Stephenson).
We reformulate these problems into a Euclidean context, which allows more
general discrete conformal structures and boundary conditions. The main idea
is to replace the disk with a double covered disk with one side forced to be a
circle and the other forced to have interior curvature zero. The entire
problem is reduced to finding a zero curvature structure. We also show that
these curvatures arise naturally as curvature measures on generalized
manifolds (manifolds with multiplicity) that extend the usual discrete
Lipschitz-Killing curvatures on surfaces.

\end{abstract}
\maketitle

\section{Introduction}

The Riemann Mapping theorem states that any simply connected, open, proper
subset of $\mathbb{C}$ can be conformally mapped to $\mathbb{D}$, and that
this map is unique up to Mobius transformations that fix the disk. Thurston
formulated a way to approximate the mapping guaranteed by the theorem (for
precompact domains in $\mathbb{C}$), and his formulation was proven to
converge to that mapping by Rodin and Sullivan \cite{RS}. Thurston's discrete
conformal mappings are based on mappings of circle packings, and central to
the argument is a theorem that a circle packing of a simply connected, bounded
domain can be deformed to a circle packing of $\mathbb{D}$ while maintaining
the tangency pattern but allowing the circle radii to change. We call such
theorems \emph{discrete uniformization theorems}. These theorems
are based the existence of circle packings on spheres, the original being the Koebe theorem (rediscovered by Andreev and generalized by Thurston):
\begin{theorem}
Let $K$ be a triangulation of the $2$-sphere that is not simplicially equivalent to the
tetrahedron. Then there is a realization of $K$ as a geodesic triangulation of 
the sphere with circles at the vertices such that the circles corresponding to two
vertices on an edge are externally tangent to each other.
\end{theorem}
\noindent The original proofs of this
theorem and some of its generalizations, for instance where the circles have prescribed intersection angles or they appear on other surfaces, can be found in \cite{Koebe, Andr, Andr2, Thurs, MR, GLSW}.

Beardon and Stephenson \cite{BeaSte} (see also \cite{Ste}) were able to
reprove the circle packing theorem via a variant of the Perron method. The
Perron method is a method of solving a partial differential equation on the
interior of the disk (the equation is that the curvature is zero) subject to
the condition that the boundary circles are internally tangent to the boundary
of the disk. The method relies on transferring the problem to a problem of
packing hyperbolic circles. This idea is especially elegant because it turns
the boundary condition to the condition that the boundary circles are
horocycles, i.e., circles of infinite radius. Another advantage to the
formulation of the problem as a problem in hyperbolic geometry is that the
nonuniqueness of the Riemann Mapping Theorem falls out because Mobius
transformations are isometries in hyperbolic space. Bowers and Stephenson were
later able to extend the argument to other boundary conditions by considering
curves of constant curvature \cite{BowSte1}. In particular, they prove
\begin{theorem}[\cite{BowSte1}, Proposition 1]
Let $K$ be a a triangulation of the disk and $\Phi$ be an assignment of values in $[0,\pi/2]$
to the edges of $K$. Given a proper choice of curvatures for boundary vertices, there exists
an assignment of curvatures to interior vertices such that
\begin{enumerate}
\item there is a drawing of 
circles in the hyperbolic plane such that two circles corresponding to two vertices joined by an edge
intersect at the angle prescribed by $\Phi$, 
\item these edges can be given hyperbolic lengths based on the hyperbolic distances between circle centers, and the sum of the angles
of these triangles at
any interior vertex is equal to $2\pi$, and
\item the circles corresponding to boundary vertices have the assigned curvatures.
\end{enumerate}
\end{theorem}
\noindent This theorem, for instance, handles the case that boundary circles intersect the 
unit circle (which is the infinity of the disk model of hyperbolic space) at right angles by
specifying that the boundary circles are geodesics (curvature is zero).

The goal of the present work is to reformulate the problem of discrete
uniformization theorems of disks in terms of Euclidean discrete conformal
structures. In light of the facts that the hyperbolic formulation is
especially elegant and that there has been recent progress on understanding
hyperbolic discrete conformal structures in great generality (e.g., 
\cite{Guo, BPS, GT, ZGZLYG}), one may question the advantage of
reformulating in terms of Euclidean structures. There are several reasons why
having a comprehensive Euclidean treatment could be useful. The primary
advantage is that the Euclidean case can treat otherwise difficult cases such
as the multiplicative discrete conformal structures studied in \cite{Luo1} and
\cite{SSP} with the boundary condition that the triangulation is internally
tangent to the unit circle (note that \cite{BPS} solves this problem by
treating the boundary condition in a different way). The multiplicative
structure is well motivative, and has close connections with the finite
element Laplacian (see, e.g., \cite{G3}). Another advantage of the Euclidean
background structure is that it is more intuitive, making it more accessible
for applications than hyperbolic formulations. Finally, this formulation of a
discrete uniformization theorem motivates the definitions at the end of this
paper on manifolds with multiplicity, giving a notion of how one might try to
use geometric flows on the interior of two disks glued together to find
uniform structures such as the Riemann mapping transformation (note that
Brendle has a slightly different formulation of a geometric flow to this end
in \cite{Bren}).

\begin{acknowledgement}
The author was partially funded by NSF grant DMS 0748283. The author would
like to thank Ken Stephenson and Joseph Fu for helpful conversations.
\end{acknowledgement}

\section{Discrete conformal structures and M-weighted points}

Discrete conformal structures are generalizations of circle packings, circles
with fixed intersection angle, circles with fixed inversive distance, and
multiplicative structures studied in papers such as 
\cite{Thurs, MR, CL, BowSte, Guo, Luo1, SSP}. Discrete
conformal structures were described from an axiomatic framework in \cite{G5}.
In \cite{GT} it is shown that discrete conformal structures have the form
$C_{\alpha,\eta},$ where the length of an edge $\ell_{ij}=\ell\left(
v_{i},v_{j}\right)  $ is given by
\[
\ell_{ij}^{2}=\alpha_{i}e^{2f_{i}}+\alpha_{j}e^{2f_{j}}+2\eta_{ij}%
e^{f_{i}+f_{j}}%
\]
where the $\alpha$ and $\eta$ are part of the conformal structure and the $f$
choose the particular metric in the conformal class. Technically the conformal
structure $C_{\alpha,\eta}$ is a map from a label $f$ to a discrete metric.
However, we will use $f\in C_{\alpha,\eta}$ to mean that $f$ is a label that
determines geometry using the conformal structure $C_{\alpha,\eta}.$ Conformal
structures were generalized to piecewise Euclidean (or polyhedral) manifolds
in \cite{GLSW} by considering only the Delaunay triangulations and our case
where the $\alpha$ are all zero.

Discrete conformal structures are associated to triangulated surfaces, but we
will need to compare these geometric structures with triangulations of points
in the plane that have certain decorations or weights. This is quite common in
the literature on weighted Delaunay triangulations (e.g., \cite{AK, ES, G3}). 
We will define a generalization of weighted points in
order to more carefully describe the setting for discrete conformal
uniformization of Euclidean disks.

\begin{definition}
\label{def: weighted point}A weighted point $\left(  p,W\right)  $ in the
plane is a point $p\in\mathbb{R}^{2}$ together with a weight $W\in\mathbb{R}$.
If $W\geq0,$ we denote the disk $D\left(  p,W\right)  =\left\{  x\in
\mathbb{R}^{2}:\left\vert p-x\right\vert ^{2}\leq W\right\}  $ and the circle
$C\left(  p,W\right)  =\left\{  x\in\mathbb{R}^{2}:\left\vert p-x\right\vert
^{2}=W\right\}  .$
\end{definition}

If $W>0,$ then we can think of the weighted point as a disk centered at $p$
with radius $\sqrt{W}.$ If $W=0,$ we think of the weighted point simply as a
point. When $W>0,$ there are several notions of \textquotedblleft
distance\textquotedblright\ between two weighted points $\left(  p_{1}%
,W_{1}\right)  $ and $\left(  p_{2},W_{2}\right)  ,$ including the distance
between the points $\left\vert p_{1}-p_{2}\right\vert $ and the Laguerre
(power) distance $\left\vert p_{1}-p_{2}\right\vert ^{2}-W_{1}-W_{2}.$ There
are also notions of angle between two intersecting weighted points (circle
intersection angles), inversive distance between nonintersecting weighted
points, and other notions. In order to make these more precise and to
generalize to when $W<0,$ we will expand the notion of a weighted point to an
M-weighted point. The M-weighted points correspond to the caps in \cite{Wil}.

\begin{definition}
Given $\xi=\left(  a,b,c,d\right)  $ and $\zeta=\left(  x,y,z,w\right)  $ in
$\mathbb{R}^{4},$ the Minkowski product is
\[
\xi\ast\zeta=ax+by+cz-dw.
\]

\end{definition}

We now define M-weighted points, which is short for Minkowski-weighted points.

\begin{definition}
Given a point $\xi=\left(  \xi^{1},\xi^{2},\xi^{3},\xi^{4}\right)
\in\mathbb{R}^{4}\setminus\left\{  0\right\}  $ with $\xi^{3}<\xi^{4},$ we may
define the map $p\left(  \xi\right)  =\left(  \frac{\xi^{1}}{\xi^{4}-\xi^{3}%
},\frac{\xi^{2}}{\xi^{4}-\xi^{3}}\right)  $ and $W\left(  \xi\right)
=\frac{\xi\ast\xi}{\left(  \xi^{4}-\xi^{3}\right)  ^{2}}.$ We call $\xi$ an
\emph{M-weighted point}. An M-weighted point $\xi$ has a corresponding
weighted point $\left(  p\left(  \xi\right)  ,W\left(  \xi\right)  \right)  $.
We use $\mathbb{R}_{\bot}^{4}$ to denote the points $\xi\in\mathbb{R}^{4}$
satisfying $\xi^{3}<\xi^{4},$ the set of possible $M$-weighted points.
\end{definition}

\begin{remark}
We may extend this map to a map $\mathbb{R}^{4}\setminus\left\{  0\right\}
\rightarrow S^{2}\times S^{1}$ using the one point compactification
formulations of $S^{2}$ and $S^{1}.$ The first component is essentially the
Hopf map (if we restrict to unit vectors in $\mathbb{R}^{4}$).
\end{remark}

The preimage of a weighted point $\left(  p,W\right)  $ is a line of points in
$\mathbb{R}^{4}\setminus\left\{  0\right\}  .$ Often we will choose the
preimage that makes $\xi^{4}-\xi^{3}=1,$ so that for $\left(  \left(
x,y\right)  ,W\right)  $ we choose
\[
\xi=\left(  x,y,\frac{1}{2}\left(  x^{2}+y^{2}-W-1\right)  ,\frac{1}{2}\left(
x^{2}+y^{2}-W+1\right)  \right)  .
\]
Note that if we pull back the Minkowski product via rescaling to so that
$\xi^{4}-\xi^{3}=1$ we get the Klein model for hyperbolic 3-space, and we can
consider (1-point compactified) Euclidean 2-space as the boundary at infinity
of hyperbolic 3-space.

An M-weighted point can sometimes be interpreted as a disk or point, and this
leads to geometric interpretations for the product *. The following
proposition is easily verified, and described well in the work of Wilker
\cite{Wil}. Recall $C\left(  p,W\right)  $ and $D\left(  p,W\right)  $ as in
Definition \ref{def: weighted point}.

\begin{proposition}
Let $\xi,\zeta\in\mathbb{R}_{\bot}^{4}.$ The following are immediate
consequences of Theorem 8 in \cite{Wil}.

\begin{itemize}
\item The unit disk corresponds to the point $\left(  0,0,1,0\right)  $ (or
any point $\left(  0,0,t,0\right)  $). We denote $U=\left(  0,0,1,0\right)  .$

\item We always have
\[
\left\vert p\left(  \xi\right)  -p\left(  \zeta\right)  \right\vert
^{2}=\left(  \frac{\xi}{\xi^{4}-\xi^{3}}-\frac{\zeta}{\zeta^{4}-\zeta^{3}%
}\right)  \ast\left(  \frac{\xi}{\xi^{4}-\xi^{3}}-\frac{\zeta}{\zeta^{4}%
-\zeta^{3}}\right)
\]

\item If $\xi\ast\xi>0$ and $\zeta\ast\zeta>0$ and $\xi\ast\zeta\leq1$ then
the circles $C\left(  p\left(  \xi\right)  ,W\left(  \xi\right)  \right)  $
and $C\left(  p\left(  \zeta\right)  ,W\left(  \zeta\right)  \right)  $
intersect with angle $\theta$ satisfying
\[
\cos\theta=\frac{\xi\ast\zeta}{\sqrt{\xi\ast\xi}\sqrt{\zeta\ast\zeta}}.
\]

\item If $\xi\ast\xi>0$ and $\zeta\ast\zeta>0$ and $\xi\ast\zeta>1$ then the
disks $D\left(  p\left(  \xi\right)  ,W\left(  \xi\right)  \right)  $ and
$D\left(  p\left(  \zeta\right)  ,W\left(  \zeta\right)  \right)  $ are
disjoint and have inversive distance $\delta$ satisfying%
\[
\cosh\delta=\frac{\xi\ast\zeta}{\sqrt{\xi\ast\xi}\sqrt{\zeta\ast\zeta}}.
\]

\item If $\xi\ast\xi=0$ and $\zeta\ast\zeta>0$ then the circle $C=C\left(
p\left(  \zeta\right)  ,W\left(  \zeta\right)  \right)  $ and the point
$p=p\left(  \xi\right)  $ satisfy
\[
P_{C}\left(  p\right)  =\left\vert p-p\left(  \zeta\right)  \right\vert
^{2}-W\left(  \zeta\right)  =\frac{\xi\ast\zeta}{\sqrt{\zeta\ast\zeta}}.
\]
where $P_{C}$ is the power (see, e.g., \cite{AK}) In particular, $p\in C$ if
and only if $\xi\ast\zeta=0.$

\item If $\xi\ast\xi=0$ and $\zeta\ast\zeta=0$ then
\[
\left\vert p\left(  \xi\right)  -p\left(  \zeta\right)  \right\vert ^{2}%
=\frac{\xi\ast\zeta}{\left(  \xi^{4}-\xi^{3}\right)  \left(  \zeta^{4}%
-\zeta^{3}\right)  }%
\]

\end{itemize}
\end{proposition}

We can also consider triangulations in the plane by M-weighted points.

\begin{definition}
A \emph{M-weighted triangulation} $\left(  T,\xi\right)  $ of a planar region
$D$ is a triangulation $T$, together with a map
\[
\xi:V\left(  T\right)  \rightarrow\mathbb{R}_{\bot}^{4},
\]
where we will use $V\left(  T\right)  ,$ $E\left(  T\right)  ,$ and $F\left(
T\right)  $ to denote the vertices, edges, and faces of $T,$ respectively.
\end{definition}

The Mobius group acts naturally on an M-weighted triangulation (see \cite{Wil}).

\begin{proposition}
\label{prop: groups}The Mobius group acts on $\left(  T,\xi\right)  $ by
linear maps $L$ such that for each $v$ and $w,$ $\left(  L\xi_{v}\right)
\ast\left(  L\xi_{w}\right)  =\xi_{v}\ast\xi_{w}.$ There is a subgroup of the
Mobius group corresponding the the Euclidean transformations of $p\left(
\xi\right)  $ that preserve the weights $W\left(  \xi\right)  .$
\end{proposition}

The M-weighted points with the connectivity of the triangulation induce a
discrete conformal structure. Note that throughout this paper, we will assume
triangulations are proper, in the sense that each edge has two distinct
vertices and that there is at most one edge connecting two vertices. Thus we
can refer to an edge $vw$ between vertices $v$ and $w$ without confusion. Much
of the work is independent of this restriction, but the statements are much
clearer in this form.

\begin{proposition}
\label{Prop:conf and mweighted}Let $T$ be a triangulation of a plane region
with vertex points $\left\{  p_{v}\right\}  _{v\in V\left(  T\right)  }$ in
the plane. For each $v\in V\left(  T\right)  $ and $vw\in E\left(  T\right)  $
fix $\alpha_{v}$ and $\eta_{vw}.$ There is a bijection between M-weighted
triangulations $\left(  T,\xi\right)  $ such that
\begin{align*}
\alpha_{v}  &  =\xi_{v}\ast\xi_{v},\\
\eta_{vw}  &  =-\xi_{v}\ast\xi_{w},\text{ and }\\
p\left(  \xi_{v}\right)   &  =p_{v}%
\end{align*}
and conformal structures $C_{\alpha,\eta}$ on $T$ such that $\ell
_{vw}=\left\vert p_{v}-p_{w}\right\vert .$
\end{proposition}

\begin{proof}
Given $\xi$, we recall that
\begin{align*}
\left\vert p_{v}-p_{w}\right\vert ^{2}  &  =\left(  \frac{\xi_{v}}{\left(
\xi_{v}^{4}-\xi_{v}^{3}\right)  }-\frac{\xi_{w}}{\left(  \xi_{w}^{4}-\xi
_{w}^{3}\right)  }\right)  \ast\left(  \frac{\xi_{v}}{\left(  \xi_{v}^{4}%
-\xi_{v}^{3}\right)  }-\frac{\xi_{w}}{\left(  \xi_{w}^{4}-\xi_{w}^{3}\right)
}\right) \\
&  =\alpha_{v}\left(  \xi_{v}^{4}-\xi_{v}^{3}\right)  ^{-2}+\alpha_{w}\left(
\xi_{w}^{4}-\xi_{w}^{3}\right)  ^{-2}+2\eta_{vw}\left(  \xi_{v}^{4}-\xi
_{v}^{3}\right)  ^{-1}\left(  \xi_{w}^{4}-\xi_{w}^{3}\right)  ^{-1}%
\end{align*}
and so we find that $\xi$ corresponds to $f_{v}=-\log\left(  \xi_{v}^{4}%
-\xi_{v}^{3}\right)  .$ Conversely, given $f\in C_{\alpha,\eta},$ we can take
\[
\xi_{v}=\left(  e^{f_{v}}p_{v},\frac{e^{f_{v}}}{2}\left(  \left\vert
p_{v}\right\vert ^{2}-\alpha_{v}e^{2f_{v}}-1\right)  ,\frac{e^{f_{v}}}%
{2}\left(  \left\vert p_{v}\right\vert ^{2}-\alpha_{v}e^{2f_{v}}+1\right)
\right)  .
\]

\end{proof}

\section{Conformal uniformization of the Euclidean disk}

\subsection{Formulation of the problem}

We will follow the notation in Stephenson's book \cite{Ste}. Recall the
following definition.

\begin{definition}
A combinatorial closed disk $S$ is a simplicial $2$-complex that triangulates
a topological closed disk, i.e., $S$ is finite, simply connected, and has
nonempty boundary. We use $\mathring{S}$ to denote the interior of the disk
and $\partial S$ to denote the boundary.
\end{definition}

We know that any geometric realization of the boundary of $S$ must be
homeomorphic to the circle. The notation is such that the vertices of $S$ are
partitioned into $V\left(  \mathring{S}\right)  $ and $V\left(  \partial
S\right)  .$

We can now state the problem of discrete conformal uniformization of a
Euclidean disk.

\begin{problem}
\label{packing problem 1}Let $U\in\mathbb{R}^{4}$ denote the representation of
the unit disk as an M-weighted point. Let $S$ be a combinatorial closed disk
and let $C_{\alpha,\eta}$ be a conformal structure and $\mu:V\left(  \partial
S\right)  \rightarrow\mathbb{R}$. Then find $\xi:V\left(  S\right)
\rightarrow\mathbb{R}_{\bot}^{4}$ such that
\begin{align*}
\xi_{v}\ast\xi_{v}  &  =\alpha_{v}\\
\xi_{v}\ast\xi_{w}  &  =\eta_{vw}%
\end{align*}
for all $v\in V\left(  S\right)  $ and all $vw\in E\left(  S\right)  ,$ and
\[
\xi_{v}\ast U=\left(  \xi_{v}^{4}-\xi_{v}^{3}\right)  \mu_{v}%
\]
for all $v\in V\left(  \partial S\right)  .$
\end{problem}

While this formulation looks a bit mysterious, consider the following cases:

\begin{itemize}
\item If we choose a conformal structure such that $\alpha_{v}=1$ for all
vertices $v\in V\left(  S\right)  $, $\eta_{e}=1$ for all $e\in E\left(
S\right)  ,$ and $\mu_{v}=-1$ for all $v\in$ $V\left(  \partial S\right)  ,$
this corresponds to circle packings with boundary circles internally tangent
to the disk (see \cite{Ste} for a fairly comprehensive coverage of this
well-studied case).

\item If we choose a conformal structure such that $\alpha_{v}=1$ for all
vertices $v\in V\left(  S\right)  $, $\eta_{e}=1$ for all $e\in E\left(
S\right)  ,$ and $\mu_{v}=0$ for all $v\in$ $V\left(  \partial S\right)  ,$
this corresponds to circle packings with boundary circles orthogonal to the
unit disk.

\item If we choose a conformal structure such that $\alpha_{v}=0$ for all
vertices $v\in V\left(  S\right)  $, $\eta_{e}>0$ for all $e\in E\left(
S\right)  ,$ and $\mu_{v}=0$ for all $v\in$ $V\left(  \partial S\right)  $,
this corresponds to triangulations with boundary points on the unit circle
with the multiplicative conformal structure.
\end{itemize}

Particular cases of this problem have been solved by Andreev-Thurston
\cite{Thurs}, Beardon-Stephenson \cite{BeaSte}, and Bowers-Stephenson
\cite{BowSte1}. The reformulation for the multiplicative conformal structure
with boundary vertices on the unit circle as also handled in \cite{BPS} in a
different way by prescribing boundary angles on a half plane.

\subsection{Reformulation as a curvature problem}

We return to the case of an abstract triangulated surface with a discrete
conformal structure. Recall we can define curvature at an interior vertex as follows.

\begin{definition}
The curvature $K_{v}$ at an interior vertex $v$ is defined to be
\[
K_{v}=2\pi-\sum_{f}\theta\left(  v<f\right)
\]
where $f$ runs over all faces and $\theta\left(  v<f\right)  $ is the angle at
vertex $v$ in face $f$ (understood to be zero if the face $f$ does not contain
$v$).
\end{definition}

We may use a monodromy theorem to see how triangulations labels giving zero
curvature correspond to M-weighted points in the plane.

\begin{theorem}
[Monodromy]\label{monodromy theorem 1}If $f\in C_{\alpha,\eta}$ is a label
such that $K_{v}=0$ for all $v\in V\left(  \mathring{S}\right)  $ there exist
points $p:V\left(  S\right)  \rightarrow\mathbb{R}^{2}$ such that
$\ell_{vv^{\prime}}=\left\vert p\left(  v\right)  -p\left(  v^{\prime}\right)
\right\vert $ for any edge $e=vv^{\prime}.$ The map $P$ is unique up to
Euclidean isometry. In fact, there exist $\xi:V\left(  S\right)
\rightarrow\mathbb{R}_{\bot}^{4}$ such that $p\left(  \xi\right)  =p$ and
$\xi_{v}\ast\xi_{v}=\alpha_{v}$ and $\xi_{v}\ast\xi_{w}=\eta_{vw}$ for edges
all $vw\in E\left(  S\right)  .$
\end{theorem}

\begin{proof}
[Proof (sketch)]We first place a face. This is unique up to Euclidean motion.
We can now successively place neighboring faces in a uniquely determined
place. We know that each vertex is placed since $S$ is connected. One can now
develop along chains (as in \cite[Theorem 5.4]{Ste}), and use the curvature is
zero condition to show it is independent of the chain. The placement is unique
once the first triangle is placed, and so is unique up to Euclidean isometry.
The last statement follows from Proposition \ref{Prop:conf and mweighted}.
\end{proof}

\begin{definition}
If $f\in C_{\alpha,\eta}$ is a label such that $K_{v}=0$ for all $v\in
V\left(  \mathring{S}\right)  $ we call $f$ a \emph{flat label}.
\end{definition}

Using the monodromy theorem, we can associate flat labels with triangulations
of points in the plane. Together with Proposition
\ref{Prop:conf and mweighted}, we get the following.

\begin{corollary}
There is a one-to-one correspondence between the set of flat labels in
$C_{\alpha,\eta}$ for a combinatorial closed disk $S$ and the equivalence
class of M-weighted triangulations $\left(  S,\xi\right)  $ where $\xi$ is
considered up to Euclidean transformations (c.f., Proposition
\ref{prop: groups}).
\end{corollary}

It follows that Problem 1 can almost be reformulated into a problem of finding
a flat label in a conformal class. However, we do not yet know how to describe
the boundary condition. The most natural boundary condition for such a problem
is to specify the weights $f$ on the boundary (Dirichlet type condition) or
specify the angle sum at the boundary (Neumann type boundary condition) as
described in \cite{BPS}. One might think that by considering the disk $U$ as
an additional weighted point, one can consider this more like an equation on a
closed manifold, and this is essentially what we will do. We will need to
describe how to add the additional point as an augmentation.

\begin{definition}
An \emph{augmented conformal combinatorial closed disk} $\hat{S}$ is a
simplicial complex determined by the following properties:

\begin{itemize}
\item There is a conformal combinatorially closed disk $S\subset\hat{S}.$

\item There is a vertex $\hat{v}\in\hat{S}\setminus S$ such that $V\left(
\hat{S}\right)  =V\left(  S\right)  \cup\left\{  \hat{v}\right\}  .$

\item For each $v\in\partial S,$ there is an edge $\hat{e}=v\hat{v};$
furthermore, we have $E\left(  \hat{S}\right)  =E\left(  S\right)
\cup\left\{  \hat{e}=v\hat{v}:v\in\partial S\right\}  .$

\item $\hat{S}$ has a discrete conformal structure and a label in that
conformal structure.
\end{itemize}
\end{definition}

Note the following.

\begin{proposition}
The vertices $V\left(  \hat{S}\right)  $ are partitioned into $V\left(
\mathring{S}\right)  ,$ $V\left(  \partial S\right)  ,$ and $\left\{  \hat
{v}\right\}  .$ The faces $F\left(  \hat{S}\right)  $ are partitioned into
$F\left(  S\right)  $ and $F\left(  \hat{S}\setminus\mathring{S}\right)  .$
\end{proposition}

In general, we will consider the augmentation of a conformal combinatorial
closed disk to an augmented conformal combinatorially closed disk in order to
impose boundary conditions on the determination of a conformal combinatorial
closed disk with certain properties.

Note that $\hat{S}$ is topologically a sphere, but we will not be considering
it in this way, since geometrically it will be far from a sphere. In
particular, the curvature of an augmented conformal combinatorial closed disk
$\hat{S}$ is slightly different than it would be for a sphere:

\begin{definition}
Let $\hat{S}$ be an augmented conformal combinatorial closed disk. The
curvatures $K:V\left(  \hat{S}\right)  \rightarrow\mathbb{R}$ are defined to
be
\[
K\left(  v\right)  =\left\{
\begin{array}
[c]{cl}%
2\pi-\sum_{f\in F\left(  S\right)  }\theta\left(  v<f\right)  & \text{if }v\in
V\left(  \mathring{S}\right)  ,\\
\sum_{f\in F\left(  \hat{S}\setminus\mathring{S}\right)  }\theta\left(
v<f\right)  -\sum_{f\in F\left(  S\right)  }\theta\left(  v<f\right)  &
\text{if }v\in V\left(  \partial S\right)  ,\\
\sum_{f\in F\left(  \hat{S}\right)  }\theta\left(  v<f\right)  -2\pi &
\text{if }v=\hat{v}.
\end{array}
\right.
\]
We say that $\hat{S}$ is flat if $K\left(  v\right)  =0$ for all $v\in
V\left(  \hat{S}\right)  .$ If the geometry comes from a label in a conformal
class, we call such a label a flat label.
\end{definition}

The motivation for this definition, is that we want to consider $\hat{S}$ as
one disk folded on the other (see Section \ref{sect: curvature measures}).
Note that zero curvature around the boundary means that the complex folds over
on itself perfectly. We may prove a restriction on the curvatures.

\begin{proposition}
The curvatures satisfy
\[
\sum_{v\in V\left(  \hat{S}\right)  }K\left(  v\right)  =0.
\]

\end{proposition}

\begin{proof}
This follows from the fact that the sum is equal to
\begin{align*}
\sum_{v\in V\left(  \hat{S}\right)  }K\left(  v\right)   &  =2\pi\left\vert
V\left(  \mathring{S}\right)  \right\vert -\pi\left\vert F\left(  S\right)
\right\vert +\pi\left\vert F\left(  \hat{S}\setminus\mathring{S}\right)
\right\vert -2\pi\\
&  =2\pi\left\vert V\left(  \mathring{S}\right)  \right\vert -\pi\left\vert
F\left(  S\right)  \right\vert +\pi\left\vert V\left(  \partial S\right)
\right\vert -2\pi\\
&  =2\pi\left\vert V\left(  S\right)  \right\vert -\pi\left\vert F\left(
S\right)  \right\vert -\pi\left\vert E\left(  \partial S\right)  \right\vert
-2\pi\\
&  =2\pi\chi\left(  S\right)  +2\pi\left\vert E\left(  S\right)  \right\vert
-3\pi\left\vert F\left(  S\right)  \right\vert -\pi\left\vert E\left(
\partial S\right)  \right\vert -2\pi\\
&  =0,
\end{align*}
since $3F\left(  S\right)  =2E\left(  S\right)  -E\left(  \partial S\right)  $
and $\chi\left(  S\right)  =1$.
\end{proof}

\begin{remark}
If instead of this definition, one only used the usual curvature as angle
deficit (first part of the curvature definition above) for every vertex, one
would get $\sum_{v\in V\left(  \hat{S}\right)  }K\left(  v\right)  =4\pi,$
following from the Euler characteristic of the sphere.
\end{remark}

We will give some basic examples:

\begin{example}
[Circle packing with internally tangent boundary]Circle packing with
internally tangent boundary. A circle packing arises from $\alpha_{v}=1$ and
$\eta_{e}=1$ for vertices and edges in the disk. In order to get boundary
circles to be internally tangent to a circle, we specify that $\alpha_{\hat
{v}}=1$ for the augmented vertex and $\eta_{e}=-1$ for any augmented edges.
This assures that $\ell_{v\hat{v}}=e^{f_{\hat{v}}}-e^{f_{v}}$ as long as
$f_{\hat{v}}\geq f_{v}.$
\end{example}

\begin{example}
[Circle packing orthogonal to boundary]In this case, we keep the circle
packing for the disk, but must specify that $\alpha_{\hat{v}}=1$ and $\eta
_{e}=0$ for any augmented edges.
\end{example}

\begin{example}
[Inscribed triangulation with multiplicative weights]Multiplicative weights
arise when $\alpha_{v}=0$ (and $\eta_{e}$ are all fixed to positive numbers).
If we specify this for all vertices of the disk, we can set $\alpha_{\hat{v}%
}=1$ (or any other positive number) and specify that $\eta_{e}=0$ for all
augmented edges. This assures that for all augmented edges, $\ell_{e}=1,$ and
so the boundary vertices lie on the unit circle. This is particularly nice
since if one can find a zero curvature solution, the triangulated disk is
inscribed in the unit circle. In this setting, one needs to make sure that the
triangle inequality is satisfied for any given choice of $\eta$'s and $f$'s.
\end{example}

We may now extend the monodromy theorem (Theorem \ref{monodromy theorem 1}) to
augmented disks.

\begin{theorem}
[Monodromy]\label{Monodromy Theorem 2}If $\vec{f}$ is a flat label, there
exist points $P:V\left(  \hat{S}\right)  \rightarrow\mathbb{R}^{2}$ such that
$\ell_{vv^{\prime}}=\left\vert P\left(  v\right)  -P\left(  v^{\prime}\right)
\right\vert $ for any edge $e=vv^{\prime}.$ The map $P$ is unique up to
Euclidean isometry.
\end{theorem}

\begin{proof}
[Proof (sketch)]In our setting, we first place $\hat{v}$ at the origin, i.e.,
$P\left(  \hat{v}\right)  =0.$. Then place a face incident on $\hat{v},$
giving placement for the other two vertices. This is unique up to rotation.
Now one can place the other faces around the vertex $P\left(  \hat{v}\right)
=0.$ Since the complex looks like a disk around $\hat{v}$ and since the
curvature at $\hat{v}$ is zero, this will give a consistent choice of
$P\left(  v\right)  $ for all $v\in V\left(  \partial S\right)  .$ Now, given
an edge in $\partial S,$ we can now place the triangles with these edges, and
we do this so that the triangles are folded back toward the center of the
disk. We can place triangles in order around the vertex $v\in V\left(
\partial S\right)  ,$ and since the curvature $K\left(  v\right)  =0,$ we must
have that the faces placed around the vertex give a consistent value for $P.$
Now we can continue to place triangles inside, this time not folding back.

We know that each vertex is placed since $S$ is connected. One can now develop
along chains (as in \cite[Theorem 5.4]{Ste}), and use the curvature is zero
condition to show it is independent of the chain. The placement is unique up
to the placement of $\hat{v}$ and rotation, which is clearly the same thing as
being unique up to Euclidean isometry.
\end{proof}

This allows us to reformulate Problem \ref{packing problem 1} into the
following equivalent problem.

\begin{problem}
\label{packing problem 2}Let $S$ be a combinatorial closed disk and let
$C_{\alpha,\eta}$ be a conformal structure and $\mu:V\left(  \partial
S\right)  \rightarrow\mathbb{R}$. Let $\hat{\alpha}_{v}=\alpha_{v}$ if $v\in
V\left(  S\right)  $ and $\hat{\alpha}_{\hat{v}}=1.$ Let $\hat{\eta}_{e}%
=\eta_{e}$ if $e\in E\left(  S\right)  $ and $\hat{\eta}_{\hat{v}v}=\mu_{v}$
for all $v\in\partial S.$ Then find a flat label $f\in C_{\hat{\alpha}%
,\hat{\eta}}\left(  \hat{S}\right)  $.
\end{problem}

We have shown the following.

\begin{theorem}
Problem \ref{packing problem 1} and Problem \ref{packing problem 2} are equivalent.
\end{theorem}

Problems \ref{packing problem 1} and \ref{packing problem 2} are the direct
generalizations of the problems described in \cite{BowSte1}. In the latter,
the formulation is done in hyperbolic geometry instead of Euclidean, which has
the advantage of allowing the assignment of curvatures (generalizing radii) of
generalized circles (called cycles there) and formulating the problem as a
boundary value problem where the boundary data is specified (a Dirichlet type
problem). However, the case of $\alpha_{v}=0$ does not appear to exist in that
context. In addition, this formulation gives an alternative that uses
primarily Euclidean instead of hyperbolic geometry. The formulation in
\cite{BowSte1} with the restrictions there allow one to prove existence and
uniqueness (see also \cite{MR} and \cite{RHD}), whereas the more general case
of Problems \ref{packing problem 1} and \ref{packing problem 2} remain open.

\section{The action of the Mobius group}

Recall that Mobius transformations act on $\left(  \hat{S},\xi\right)  ,$ a
M-weighted triangulation of the augmented disk. Since, by Proposition
\ref{Prop:conf and mweighted} there is a correspondence with flat labels $f\in
C_{\hat{\alpha},\hat{\eta}},$ we see that Mobius transformations must act on
the flat labels. In particular, taking a one parameter family of Mobius
transformations through the identity, we get a deformation of $f$ through flat
labels, so it is not possible that the solution to Problems
\ref{packing problem 1} and \ref{packing problem 2} is unique. However, we
wish to show that these are the only deformations.

A first question is what are the possible deformations through flat labels.

\begin{proposition}
\label{prop: mobius invariance}The infinitesimal Mobius transformations induce
the following variations of flat labels: Suppose $p_{i}$ are the points
representing the vertices. Then the variations are all of the form%
\[
\delta f_{i}=2\left(  a,b\right)  \cdot p_{i}+c,
\]
where $\left(  a,b\right)  \in\mathbb{R}^{2}$ and $c\in\mathbb{R}$.
\end{proposition}

The proof is left for the Appendix. Note that by the work in \cite{G4}, the
derivative of the curvature map $f\rightarrow K$ takes the form
\[
\frac{dK_{v}}{dt}=\triangle\frac{df_{v}}{dt}=\sum_{vw\in E\left(  T\right)
}\frac{\ell_{vw}^{\ast}}{\ell_{vw}}\left(  \frac{df_{w}}{dt}-\frac{df_{v}}%
{dt}\right)
\]
where $\ell_{vw}^{\ast}$ is the dual length of the edge $vw$ as determined by
the discrete conformal structure. Thus the Mobius transformations induce a
kernel of the Laplacian operator $\triangle$. It is interesting to observe
that the kernel appears to correspond to linear functions analogously to the
kernel of the smooth Euclidean Laplacian. We can see this directly as follows.
Consider a vertex $p_{0},$ with edges $p_{1},\ldots,p_{k}$ adjacent to it. For
each edge $p_{0}p_{i},$ there is a dual edge induced by the conformal
structure, and all the edges around the vertex correspond to a cycle. Notice
that if $R$ denotes rotation by $\pi/2$, $\frac{\ell_{p_{i}p_{0}}^{\ast}}%
{\ell_{p_{i}p_{0}}}R\left(  p_{i}-p_{0}\right)  $ is a vector of length
$\ell_{p_{i}p_{0}}^{\ast}$ in the direction of the dual edge. This sum is then
zero because of it is the sum of vectors on a cycle. Thus,
\[
0=\triangle Rp_{i}=R\triangle p_{i}.
\]

\section{Open problems}

\subsection{Uniqueness}

We have already shown that the solutions to Problems \ref{packing problem 1}
and \ref{packing problem 2} are invariant under Mobius transformations. We
conjecture that these are the only such invariant deformations.

\begin{conjecture}
Let $f\in C_{\hat{\alpha},\hat{\eta}}$ be a flat label. Then the only
deformations of $f$ through flat labels correspond to Mobius transformations.
\end{conjecture}

We already know that the Mobius transformations form a three-dimensional
family of deformations through flat labels, so we could try to use the
implicit function theorem to show that there are no more. Instead of
formulating the problem on labels, we can use M-weighted points, since there
is a correspondence between M-weighted points and flat labels. Thus it will be
sufficient to show that Mobius transformations are the only deformations of
M-weighted points $\left(  \hat{S},\xi\right)  $ that fix $\xi_{v}\ast\xi
_{v}=\alpha_{v}$ for all $v\in V\left(  \hat{S}\right)  $ and $\xi_{v}\ast
\xi_{w}=\eta_{vw}$ for $vw\in E\left(  \hat{S}\right)  .$

Before reformulating the problem using the implicit function theorem, we count
the dimensions. The number of constraints is $V\left(  \hat{S}\right)
+E\left(  \hat{S}\right)  .$ The number of variables is $4V\left(  \hat
{S}\right)  $ since there is a $\xi_{v}\in\mathbb{R}^{4}$ for each vertex.
Since $\hat{S}$ is topologically a triangulation of the sphere, we have
$3F\left(  \hat{S}\right)  =2E\left(  \hat{S}\right)  $ and Euler's formula
$V\left(  \hat{S}\right)  -E\left(  \hat{S}\right)  +F\left(  \hat{S}\right)
=2.$ It thus follows that
\[
V\left(  \hat{S}\right)  +E\left(  \hat{S}\right)  =4V\left(  \hat{S}\right)
-6.
\]
Since the Mobius transformations form a six dimensional set of deformations,
if we can show that the differential of the map of the variables to the
constraints is injective, then we are done by the implicit function theorem.

We can label the constraints first by the equations determined by vertices and
then by the equations determined by edges. For the differential of the map, we
can arrange the matrix to have the schematics as follows:
\[
\left[
\begin{array}
[c]{cccc}%
\xi_{v_{1}}^{T} & 0 & \cdots & 0\\
0 & \xi_{v_{2}}^{T} & \cdots & 0\\
0 & 0 & \ddots & 0\\
0 & 0 & \cdots & \xi_{v_{N}}^{T}\\
\xi_{v_{2}}^{T} & \xi_{v_{1}}^{T} & \cdots & 0\\
0 & 0 & \cdots & \\
\vdots & \xi_{v_{k}}^{T} & \cdots & \\
\xi_{v_{j}}^{T} & 0 & \cdots & \\
\vdots & \vdots & \vdots &
\end{array}
\right]
\]
where the first $N=V\left(  \hat{S}\right)  $ rows correspond to the vertices
and the remaining rows correspond to the edges. In the first $N$ rows, each
row has one block of four potentially nonzero entries given by the row vector
$\xi_{v}^{T},$ the transpose of the M-weighted point corresponding to vertex
$v.$ Below $\xi_{v}^{T}$ are all the $\xi_{w}^{T}$ such that $vw\in E\left(
\hat{S}\right)  .$ In the schematic above, it is assumed that $v_{1}v_{2}\in
E\left(  \hat{S}\right)  .$ The conjecture is now equivalent to showing that
this matrix has full rank.

Another natural question is what happens when the curvatures are not zero.
Brief numerical study suggests that the labels are rigid in this case.

\begin{conjecture}
Under an appropriate assumption on the conformal structure, if the curvatures
are all nonzero then the only deformations of the label that preserve the
curvature scale all of the labels equally.
\end{conjecture}

\subsection{Existence}

The existence of circle packings of a disk was proven by Koebe \cite{Koebe},
Andreev \cite{Andr, Andr2}, and a new proof was given by Marden-Rodin
\cite{MR} in the spirit of work of Thurston \cite{Thurs}. A different proof
was given by Beardon and Stephenson \cite{BeaSte} (see also \cite{BowSte, Bow, Ste}) 
that uses a method similar to the Perron method in
proving existence of solutions to elliptic PDE by considering the hyperbolic
circle packing. This has the advantage to having a unique label (since
Euclidean labels can be changed by elements of the Mobius group) and is
possible partly because internally tangent circles are horocycles. In
\cite{BowSte} it was shown that other boundary conditions (such as orthogonal
intersection with the unit circle) can be realized in the hyperbolic
background by considering constant curvature curves instead of just circles.
The main obstacles to using this method in our setting is that we do not have
strict monotonicity of the curvatures due to the effect of the unusual
boundary curvature definitions, and the space of solutions (satisfying
triangle inequality) is not convex. One may be able to deal with the latter
issue using the methods in \cite{BPS} and \cite{Luo2}, which extend locally
convex functionals continuously by constants, allowing angles to be a
continuous function of three lengths regardless of whether the triangle
inequality is satisfied.

\subsection{Computation}

One way of finding solutions is suggested by the work of Chow and Luo in the
setting of triangulated surfaces without boundary \cite{CL}. They propose a
geometric flow of radii that is essentially a gradient flow of a convex
functional. This idea was later improved by X. Gu, who suggested using
Newton's method to do the computation (see \cite{DGL}). In our setting, one
can also try Newton's method to find solutions. Although we have not yet shown
that the Jacobian matrix is nonsingular if the label is not flat, one can
still apply Newton's method using a pseudoinverse instead of inverse as in
\cite{Gay}. This may be necessary, since we know that solutions do, in fact,
have singular Jacobians due to Proposition \ref{prop: mobius invariance}.

The gradient flow idea does not work with this functional, which is clearly
not convex. However, one could try other flows that do not arise as gradient
flows. One particular choice is to take the following:
\begin{align*}
\frac{df_{v}}{dt} &  =-K_{v}~~\text{if }v\neq\hat{v},\\
\frac{df_{\hat{v}}}{dt} &  =K_{\hat{v}}.
\end{align*}
This flow has the advantage of locally looking like a heat flow on the
curvature at the interior vertices and the augmented vertex, so that it is
infinitesimally trying to make both flat. However, it is unclear what this is
doing to the boundary vertices. Preliminary study of this flow numerically is
promising, finding uniformizations relatively quickly. Figures \ref{mult},
\ref{pack 2}, and \ref{pack} show uniformizations computed using Mathematica's
NDSolve with time up to 200 starting from the triangulation in Figure
\ref{orig triang}. It is still an open problem whether the convergence is
exponential as is the case of combinatorial Ricci flow.%

\begin{figure}
[ptb]
\begin{center}
\includegraphics[
trim=-0.763780in 0.000000in -0.763780in 0.000000in,
natheight=3.359100in,
natwidth=3.467000in,
height=1.6895in,
width=2.4965in
]%
{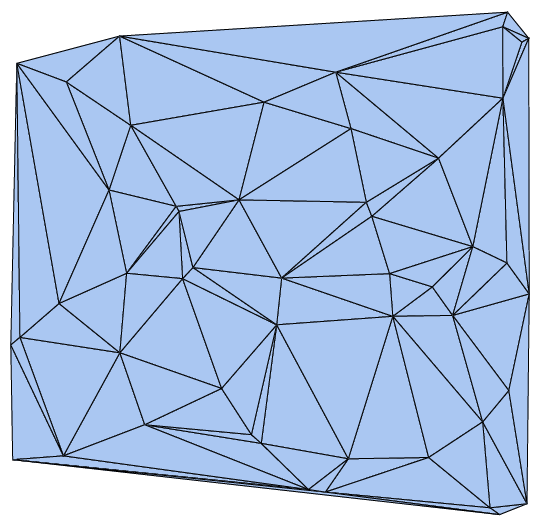}%
\caption{Original triangulation, generated as a random Delaunay
triangulation.}%
\label{orig triang}%
\end{center}
\end{figure}
\begin{figure}
[ptbptb]
\begin{center}
\includegraphics[
natheight=3.450400in,
natwidth=3.467000in,
height=1.7352in,
width=1.7426in
]%
{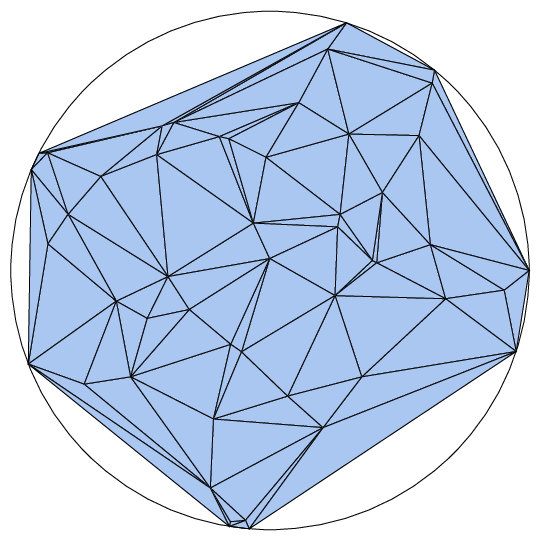}%
\caption{Multiplicative uniformization, with boundary vertices on unit
circle.}%
\label{mult}%
\end{center}
\end{figure}
\begin{figure}
[ptbptbptb]
\begin{center}
\includegraphics[
natheight=3.450400in,
natwidth=3.467000in,
height=1.7501in,
width=1.7592in
]%
{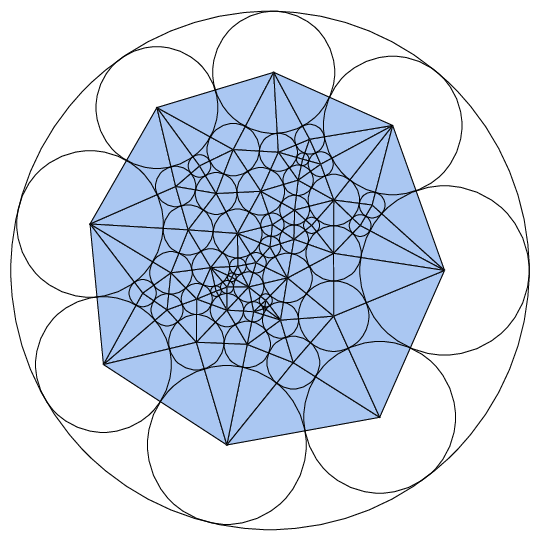}%
\caption{Circle packing uniformization, with boundary circles internally
tangent to unit circle.}%
\label{pack 2}%
\end{center}
\end{figure}
\begin{figure}
[ptbptbptbptb]
\begin{center}
\includegraphics[
natheight=3.051000in,
natwidth=3.467000in,
height=1.5508in,
width=1.7592in
]%
{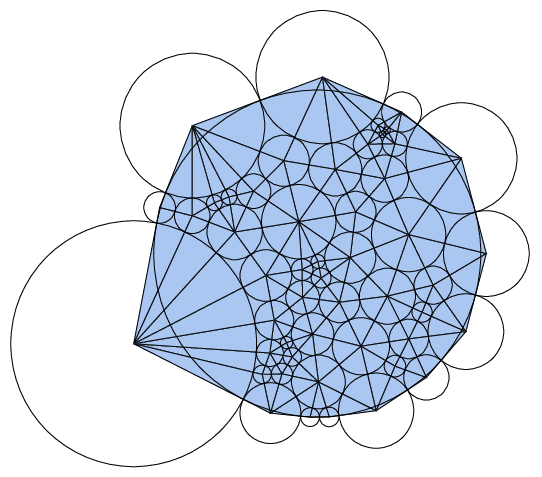}%
\caption{Circle packing uniformization, with boundary circles orthogonal to
unit circle.}%
\label{pack}%
\end{center}
\end{figure}

\section{Appendix:\ Mobius transformation computations}

We consider small perturbations of the identity by Mobius transformations. A
general perturbation of the identity is of the form
\[
I+\varepsilon M
\]
for small $\varepsilon,$ where $I$ is the identity matrix and
\[
M=\left(
\begin{array}
[c]{cccc}%
a & d & \ell & g\\
b & e & m & h\\
q & p & n & s\\
c & f & r & k
\end{array}
\right)  .
\]
For $\left(  I+\varepsilon M\right)  $ to be a Mobius transformation, we must
have
\[
\left(  I+\varepsilon M\right)  ^{T}\left(
\begin{array}
[c]{cccc}%
1 & 0 & 0 & 0\\
0 & 1 & 0 & 0\\
0 & 0 & 1 & 0\\
0 & 0 & 0 & -1
\end{array}
\right)  \left(  I+\varepsilon M\right)  =\left(
\begin{array}
[c]{cccc}%
1 & 0 & 0 & 0\\
0 & 1 & 0 & 0\\
0 & 0 & 1 & 0\\
0 & 0 & 0 & -1
\end{array}
\right)  .
\]
Looking at this up to $O\left(  \varepsilon^{2}\right)  ,$ we find that the
family takes the form
\[
\left(
\begin{array}
[c]{cccc}%
1 & \varepsilon r & -\varepsilon b & -\varepsilon a\\
-\varepsilon r & 1 & -\varepsilon d & -\varepsilon c\\
\varepsilon b & +\varepsilon d & 1 & \varepsilon t\\
-\varepsilon a & -\varepsilon c & \varepsilon t & 1
\end{array}
\right)
\]
for real numbers $a,b,c,d,t,r.$ To compute the derivative of the action of
such a family of Mobius transformations, we compute%
\[
\xi=\left(
\begin{array}
[c]{cccc}%
1 & \varepsilon r & -\varepsilon b & -\varepsilon a\\
-\varepsilon r & 1 & -\varepsilon d & -\varepsilon c\\
\varepsilon b & +\varepsilon d & 1 & \varepsilon t\\
-\varepsilon a & -\varepsilon c & \varepsilon t & 1
\end{array}
\right)  \left(
\begin{array}
[c]{c}%
x\\
y\\
\frac{1}{2}\left(  x^{2}+y^{2}-W-1\right) \\
\frac{1}{2}\left(  x^{2}+y^{2}-W+1\right)
\end{array}
\right)
\]
and we get%
\[
\xi^{4}-\xi^{3}=1-\varepsilon\left(  \left(  a+b\right)  x+\left(  c+d\right)
y+t\right)  +O\left(  \varepsilon^{2}\right)  .
\]
Since
\[
\xi^{4}-\xi^{3}=e^{-f}%
\]
we must have that
\[
\delta f=\left(  a+b,c+d\right)  \cdot\left(  x,y\right)  +t.
\]

\section{Appendix: Curvature measures of polyhedral manifolds with
multiplicities\label{sect: curvature measures}}

The curvatures we used on the boundary of the disk in the augmented disk seem
to be designed specifically for our purposes. In this section, we show how
these are related to natural curvature measures; a good reference is
\cite{Mor}. An important property of curvature measures is the valuation property.

\begin{definition}
A measure $\mu$ satisfies the \emph{valuation} property if for measurable sets
$A$ and $B,$ we have
\[
\mu\left(  A\cup B\right)  =\mu\left(  A\right)  +\mu\left(  B\right)
-\mu\left(  A\cap B\right)  .
\]

\end{definition}

The valuation property allows one to compute the measure from constituent
parts. In particular, we can compute curvature measures of polyhedral surfaces
that are not embedded, but that are made up of pieces that are each embedded.
An alternative to having the need for the negative sign is to construct unions
that have the correct multiplicity. In particular, if we assume $A\cup B$ is
constructed by first gluing $A$ and $B$. Since gluing $A$ and $B$ gives
multiplicity $2$ on $A\cap B,$ we need to pull out a copy of $A\cap B,$ or
glue in a copy of $A\cap B$ with multiplicity $-1.$ Thus, the computation of
the measure with the valuation property (in this case, curvature) is simply
the sum of the gluings considered with multiplicities.

\begin{example}
\label{surface example}A flower of a vertex is defined to be the set of faces
containing that vertex, which is said to be the center of the flower. We can
compute the curvature at the center vertex in a flower by gluing together each
constituent triangle around the flower. In order to keep multiplicity equal to
one at all points we specify that the vertex has multiplicity $1$, the
(closed) edges have multiplicity $-1,$ and the triangles have multiplicity
$1.$ In each triangle $f_{i}$ with edges $e_{j},$ we can specify the curvature
in terms of $K\left(  v,v\right)  =2\pi,$ $K\left(  v,e_{j}\right)  =-\pi$,
$K\left(  v,f_{i}\right)  =\pi-\alpha_{i},$ where $\alpha_{i}$ is the interior
angle at the vertex $v$ of triangle $f_{i}.$ These formulas arise from the
tube formulas computing the change in area or volume of small balls around the
triangle (isometrically embedded into Euclidean space of any dimension); see
\cite{Mor}. We can now sum to get the curvature at $v$ to be%
\begin{align*}
K\left(  v\right)   &  =K\left(  v,v\right)  +\sum_{e_{i}>v}K\left(
v,e_{i}\right)  +\sum_{f_{j}>v}K\left(  v,f_{j}\right) \\
&  =2\pi-\sum_{f_{j}>v}\alpha_{j},
\end{align*}
which is the usual definition for curvature at a vertex.
\end{example}

In our setting, we have taken a triangulated disk, which can be given
multiplicities as above in the interior, and attached another disk in the
augmentation. The multiplicities of each of the simplices in the augmented
disk should be as follows to allow for a total multiplicity of zero when the
disks fit on top of each other (curvatures are all zero):%

\[%
\begin{tabular}
[c]{|l|l|l|}\hline
\textbf{Location} & \textbf{Simplex type} & \textbf{Multiplicity}%
\\\hline\hline
interior & vertex & $1$\\\hline
boundary & vertex & $0$\\\hline
augmented & vertex & $-1$\\\hline
interior & edge & $-1$\\\hline
boundary & edge & $0$\\\hline
augmented & edge & $1$\\\hline
interior & face & $1$\\\hline
augmented & face & $-1.$\\\hline
\end{tabular}
\
\]
Using these multiplicities and the usual definition of curvature measure at a
vertex of a polyhedral manifold, one gets precisely the definitions given
above for curvatures.

The advantage to this definition is the relation to tube formulas (see
\cite{Mor}). The valuation property allows one to extend such tube formulas to
more general types of surfaces as the one we consider here: a folder over
disk. Here is a formulation of polyhedral manifold with multiplicities that
allows for the proper generalization of tube formulas.

\begin{definition}
A \emph{polyhedral manifold with multiplicities} is a polyhedral manifold $M$
together with a multiplicity function $\mu:\Sigma\rightarrow\mathbb{Z}$, where
$\Sigma$ is the collection of all simplices in $M.$ If the dimension of $M$ is
two, we call $M$ a surface.
\end{definition}

\begin{definition}
If $x\in M,$ the \emph{multiplicity} at $x$ is defined as
\[
\mu\left(  x\right)  =\sum_{\sigma>x}\mu\left(  \sigma\right)
\]
where the sum is over all simplices containing $x.$ If there exists $m$ such
that $\mu\left(  x\right)  =m$ for all $x\in M,$ we say the multiplicity of
$M$ is equal to $m.$
\end{definition}

We note that a polyhedral surface as described in Example
\ref{surface example} is a polyhedral manifold with multiplicity $1.$ We can
now define curvature for a polyhedral surface with multiplicities.

\begin{definition}
The \emph{curvature} $K_{v}$ at a vertex $v\in V\left(  M\right)  $ of a
polyhedral surface with multiplicities $\left(  M,\mu\right)  $ is equal to
\[
K_{v}=2\pi\mu\left(  v\right)  +\sum_{e>v}\pi~\mu\left(  e\right)  +\sum
_{f>v}\left(  \pi-\theta_{v<f}\right)  \mu\left(  f\right)  .
\]

\end{definition}

Example \ref{surface example} describes the relationship to the usual
definition of curvature on a polyhedral manifold. These manifolds seem to be
related to the theory of currents in geometric measure theory (see, e.g.,
\cite{Morg}). In the case of zero curvature, it is clear that there is a map
from the flat polyhedral manifold with multiplicities to a current (in
$\mathbb{R}^{2}$ or $\mathbb{R}^{3},$ for instance). In the case of augmented
domain, the goal is to find a polyhedral manifold with multiplicities that
maps to a current that is equivalent to the zero current (since multiplicities
are such that the top and bottom parts of the domain sum to zero everywhere).
In general, the existence of a current occurs on each hinge (a pair of
simplices sharing a codimension 1 simplex, see \cite{AK}, \cite{ES}). Hinges
have bistellar flips, and certain geometric invariants do not change with
bistellar flips (volume, the induced distance function, the curvature) while
others do (conformal variations, Laplacians). The current associated to a
hinge is a generalization of volume, and does not change with a bistellar flip.

\begin{proposition}
Any hinge can be mapped to a current. Note that the image of a hinge and its
bistellar flip are the same.
\end{proposition}

\begin{proof}
For a hinge, if the two faces have the same multiplicity, we simply mapped to
the unfolded hinge with that multiplicity. If the two have different
multiplicities, we fold the hinge over and add the multiplicities on the overlap.
\end{proof}

\begin{remark}
The curvature measures here and their generalizations to higher dimensions
have various names including Gauss-Bonnet curvatures and Lipschitz-Killing
curvatures. These curvatures appear in essentially three places: tube
formulas, kinematic formulas, and heat trace formulas. See, e.g., \cite{CMS}.
It would be interesting to better understand the relationship of these on
polyhedral manifolds with multiplicities.
\end{remark}

\begin{remark}
Marden and Rodin used a similar construction in \cite{MR} to prove the
Andreev-Thurston theorem on the sphere. If one makes the triangle have
multiplicity -1 then one can replace their formulation of having a point
$\left(  2\pi/3,2\pi/3,2\pi/3,0,\ldots,0\right)  $ in the image of $f$ with
the existence of a flat label using our definition of curvature.
\end{remark}

\end{document}